\documentclass{amsart}
\usepackage{mathptmx, bm, amssymb, amscd, enumerate, colonequals, url, color}
\definecolor{chianti}{rgb}{0.6,0,0}
\definecolor{meretale}{rgb}{0,0,.6}
\definecolor{leaf}{rgb}{0,.35,0}
\usepackage[colorlinks=true, pagebackref, hyperindex, citecolor=meretale, urlcolor=leaf, linkcolor=chianti]{hyperref}

\newtheorem{theorem}{Theorem}[section]
\newtheorem{lemma}[theorem]{Lemma}

\theoremstyle{definition}

\newtheorem{example}[theorem]{Example}
\newtheorem{remark}[theorem]{Remark}
\newtheorem{conjecture}[theorem]{Conjecture}
\newtheorem{problem}[theorem]{Problem}
\newtheorem{question}[theorem]{Question}
\numberwithin{equation}{theorem}

\def\ge{\geqslant}
\def\le{\leqslant}
\def\phi{\varphi}
\def\bar{\overline}

\def\into{\lhook\joinrel\longrightarrow}
\def\mapsto{\longmapsto}

\def\gcd{\operatorname{gcd}}
\def\height{\operatorname{height}}

\def\rad{\operatorname{rad}}
\def\Res{\operatorname{Res}}

\def\fraka{\mathfrak{a}}
\def\frakb{\mathfrak{b}}
\def\frakm{\mathfrak{m}}

\def\CC{\mathbb{C}}

\def\FFp{\mathbb{F}_{\!p}}
\def\NN{\mathbb{N}}

\def\QQ{\mathbb{Q}}

\def\ZZ{\mathbb{Z}}

\def\calZ{\mathcal{Z}}
\def\bsz{{\boldsymbol{z}}}
\def\real{\operatorname{Re}}

\begin{document}
\title{Ideals generated by power sums}

\author{Aldo Conca}
\address{Dipartimento di Matematica, Dipartimento di Eccellenza 2023-2027, Universit\'a di Genova, Italy}
\email{aldo.conca@unige.it}

\author{Anurag K. Singh}
\address{Department of Mathematics, University of Utah, 155 South 1400 East, Salt Lake City, UT~84112, USA}
\email{singh@math.utah.edu}

\author{Kannan Soundararajan}
\address{Department of Mathematics, Stanford University, 450 Serra Mall, Stanford CA~94305, USA}
\email{ksound@stanford.edu}

\thanks{A.C. is supported by PRIN~2020355B8Y ``Squarefree Gr\"obner degenerations, special varieties and related topics,'' by MIUR Excellence Department Project awarded to the Dept.~of Mathematics, Univ.~of Genova, CUP~D33C23001110001, and by INdAM-GNSAGA; A.K.S. is supported by NSF grants DMS~2101671 and DMS~2349623; K.S. is supported by a Simons Investigator award from the Simons Foundation, and by NSF grant DMS~2100933. A.C. and A.K.S. were also supported by NSF grant DMS~1928930 and by Alfred P. Sloan Foundation grant G-2021-16778, while in residence at SLMath/MSRI, Berkeley, during the Spring 2024 Commutative Algebra program.}

\dedicatory{Dedicated to the memory of Lucian B\u{a}descu}

\begin{abstract}
We consider ideals in a polynomial ring generated by collections of power sum polynomials, and obtain conditions under which these define complete intersection rings, normal domains, and unique factorization domains. We also settle a key case of a conjecture of Conca, Krattenthaler, and Watanabe, and prove other results in that direction.
\end{abstract}
\maketitle

\section{Introduction}

Let $S\colonequals K[x_1,\dots, x_n]$ be a polynomial ring over a field $K$. For a positive integer~$a$, we use~$p_a$ to denote the power sum $x_1^a+\dots+x_n^a$. If $K$ has characteristic zero and~$a_1,a_2,\dots, a_n$ are distinct positive integers, the Jacobian criterion shows that $p_{a_1}, \dots, p_{a_n}$ are algebraically independent polynomials over $K$; the problem of determining when $n+1$ power sums generate the field of symmetric rational functions in~$x_1,\dots,x_n$ over $K$ is settled in~\cite{DZ}. In a different direction, the following is studied in~\cite{CKW}:

\begin{problem}
\label{problem:ci}
Characterize the sets $A\colonequals\{a_1,a_2,\dots, a_n\}$ of positive integers such that the corresponding power sums $p_{a_1}, \dots, p_{a_n}$ form a regular sequence in the polynomial ring $S$.
\end{problem}

The base field is taken to be $\CC$ in~\cite{CKW}, but the problem makes sense more generally.

\begin{remark}
\label{remark:intro}
We record some straightforward observations; some of these are proved in~\cite{CKW} in the case $K=\CC$, but the proofs are readily adapted to the more general setting.
\begin{enumerate}[\quad\rm(1)]
\item Whether $p_{a_1}, \dots, p_{a_n}$ is a regular sequence is unaffected by enlarging $K$, so one may assume that the base field $K$ is algebraically closed.

\item Set $d\colonequals\gcd(a_1,a_2,\dots,a_n)$. It is readily seen that $p_{a_1},\dots,p_{a_n}$ is a regular sequence precisely if $p_{a_1/d},\dots,p_{a_n/d}$ is a regular sequence. Thus, in studying Problem~\ref{problem:ci}, one may assume that $\gcd(a_1,a_2,\dots,a_n)=1$.

\item A necessary condition for $p_{a_1}, \dots, p_{a_n}$ to be a regular sequence is that $n!$ divides the product $a_1a_2\cdots a_n$.

\item If the characteristic of $K$ is either $0$ or strictly greater than $n$, and $a_1,\dots,a_n$ are consecutive positive integers, then $p_{a_1}, \dots, p_{a_n}$ is a regular sequence.

\item If $p_{a_1},\dots,p_{a_n}$ form a regular sequence in $\CC[x_1,\dots, x_n]$, then they form a regular sequence in $\FFp[x_1,\dots, x_n]$ for sufficiently large prime integers $p$. However, finding optimal bounds for such primes appears hard; for example, $p_1,p_6,p_{100}$ is a regular sequence in $\CC[x_1,x_2,x_3]$, but is not a regular sequence in~$\FFp[x_1,x_2,x_3]$ for the prime integer $p=4594399$.

\item Problem~\ref{problem:ci} is easily answered for $n=2$: polynomials $p_a,p_b$ form a regular sequence in $K[x_1,x_2]$ if and only if the characteristic of $K$ differs from $2$, and either~$a/\gcd(a,b)$ or $b/\gcd(a,b)$ is even.
\end{enumerate}
\end{remark}

Problem~\ref{problem:ci} is unresolved for $n=3$; the following is \cite[Conjecture~2.10]{CKW}:

\begin{conjecture}
\label{conjecture:6abc}
Suppose $n=3$, the characteristic of the field $K$ is zero, and that $a,b,c$ are integers with $0<a<b<c$ and $\gcd(a,b,c)=1$.
Then $p_a,p_b,p_c$ is a regular sequence if and only if $6$ divides $abc$.
\end{conjecture}

One direction holds more generally, as recorded in Remark~\ref{remark:intro}. The conjecture is proven for certain special values of $a,b,c$ in \cite{CKW}; the case $a=1$ is completely settled in~\S\ref{section:a=1} of the present paper, while in~\S\ref{section:abc:general} we prove that for each fixed positive integer~$a$, there are at most finitely many triples $(a,b,c)$ that possibly violate Conjecture~\ref{conjecture:6abc}.

In \cite[Conjecture~12]{MSW} the authors extend Conjecture~\ref{conjecture:6abc} to a statement about the zero loci of $p_a,p_b,p_c$, under the assumption that $\gcd(a,b,c)=1$, and verify their conjecture computationally for $a+b+c\le 300$; we prove this stronger conjecture in the case $a=1$.

In general, for distinct integers with $\gcd(a_1,a_2,\dots,a_n)=1$ and $n!$ dividing~$a_1a_2\cdots a_n$, the elements $p_{a_1},\dots,p_{a_n}$ need not form a regular sequence. Consider for example the case where $n=4$, and take $p_{a_1},\dots,p_{a_4}$ in $S\colonequals\CC[x_1,x_2,x_3,x_4]$. Let $\nu_2$ denote the $2$-adic valuation on $\ZZ\smallsetminus\{0\}$. If each $\nu_2(a_i)$ is either $0$ or $k$, for $k$ a fixed positive integer, then
\[
(p_{a_1},\dots,p_{a_4})\ \subseteq\ \big(x_1+x_2,\ x_3+x_4,\ x_1^{2^k}+x_3^{2^k}\big),
\]
which justifies condition~\eqref{conjecture:4abcd:2} in the conjecture below. For~\eqref{conjecture:4abcd:3}, note that $p_5\in(p_1,p_2)S$ by Remark~\ref{remark:codimension}, and consequently $p_{5d}\in(p_d,p_{2d})S$ for each positive integer $d$. A similar argument shows that $p_5\in(p_1,p_3)S$, so the set $A$ does not contain a subset of the form $\{d,3d,5d\}$; this condition, however, is implied by the others. The three conditions in the conjecture below are necessary and independent, see \cite[Remark~2.16]{CKW}.

\begin{conjecture}(\cite[Conjecture~2.15]{CKW})
\label{conjecture:4abcd}
Suppose that $n=4$ and that $K$ has characteristic zero. Let $A\colonequals\{a_1,a_2,a_3,a_4\}$ where $\gcd(a_1,a_2,a_3,a_4)=1$. Then $p_{a_1}, p_{a_2}, p_{a_3}, p_{a_4} $ is a regular sequence if and only if $A$ satisfies the following conditions:

\begin{enumerate}[\quad\rm(1)]
\item\label{conjecture:4abcd:1} The product $a_1a_2a_3a_4$ is a multiple of $24$;

\item\label{conjecture:4abcd:2} the set $\{\nu_2(a_i)\mid a_i\in A\}$ contains at least two distinct positive integers;

\item\label{conjecture:4abcd:3} the set $A$ does not contain a subset of the form $\{d,2d,5d\}$ for any $d\in \NN$.
\end{enumerate}
\end{conjecture}

\section{Primality, normality, and factoriality}

The discussion thus far concerned when power sums $p_{a_1},\dots,p_{a_n}$ form a regular sequence in $K[x_1,\dots,x_n]$. It is also natural to ask:

\begin{question}
\label{question:prime:normal}
For a set of positive integers $A\colonequals\{a_1,\dots,a_c\}$, let $p_A$ denote the sequence of power sum polynomials $p_{a_1},\dots, p_{a_c}$ in $S\colonequals K[x_1,\dots,x_n]$, and let $I_A\colonequals(p_A)$ denote the corresponding ideal of $S$.
\begin{enumerate}[\quad\rm(1)]
\item When is $p_A$ a regular sequence, equivalently when is the ideal $I_A$ a complete intersection of codimension $c$?
\item When is $S/I_A$ a normal domain?
\item When is $S/I_A$ a unique factorization domain?
\item When is the ideal $I_A$ radical?
\item When is the ideal $I_A$ prime?
\end{enumerate}
\end{question}

\begin{remark}
\label{remark:codimension}
The specification ``of codimension $c$'' in (1) is relevant; in general, the elements $p_{a_1},\dots, p_{a_c}$ need not be minimal generators of $I_A$. For example, when $n\le 4$, the polynomials $p_1,p_2,p_3,p_4$ generate the ring of symmetric polynomials; degree considerations then imply that $p_5$ is a $K$-linear combination of $p_1^5$, $p_1^3p_2$, $p_1^2p_3$, $p_1p_2^2$, $p_1p_4$, and~$p_2p_3$, so~$p_5$ is an element of the ideal $(p_1,p_2)$. Hence $(p_1,p_2,p_5)=(p_1,p_2)$ is a complete intersection ideal, though not of codimension $3$. The same argument shows as well that $p_5$ must be an element of the ideal $(p_1,p_3)$.
\end{remark}

While we do not pursue it here, one may consider analogues of these questions for other families of symmetric polynomials such as complete symmetric polynomials or elementary symmetric polynomials; see for example \cite[Conjecture~2.17]{CKW}.

\begin{theorem}
\label{theorem:n:large}
For distinct positive integers $a_1,\dots,a_c$ consider the ideal $I_A\colonequals(p_{a_1},\dots, p_{a_c})$ in the polynomial ring $S\colonequals\CC[x_1,\dots,x_n]$.
\begin{enumerate}[\quad\rm(1)]
\item\label{theorem:n:large:ci} If $n\ge 2c-1$, then the ideal $I_A$ is a complete intersection of codimension $c$.

\item\label{theorem:n:large:normal} If $n\ge 2c+1$, then $S/I_A$ is a normal domain.

\item\label{theorem:n:large:ufd} If $n\ge 2c+3$, then $S/I_A$ is a unique factorization domain.

\item\label{theorem:n:large:radical} If $n\ge 2c$, then the ring $S/I_A$ is reduced.
\end{enumerate}
\end{theorem}

Before proceeding with the proof, we note that the bounds in the theorem are optimal:

\begin{example}
\label{example:n:large}
(1) Suppose $n=2c-2$, take $A\colonequals\{1,3,5,\dots,2c-1\}$. Then $|A|=c$ but the ideal $I_A$ has height at most $c-1$ since
\[
I_A\ \subseteq\ \big(x_1+x_2,\ x_3+x_4,\ \dots,\ x_{2c-3}+x_{2c-2}\big).
\]
Indeed, the height $c-1$ ideal displayed on the right contains $p_a$ for each odd integer $a$.

(2) We show that $I_A$ need not be prime in the case $n=2c$. If $c=1$, the ideal $(p_2)$ is not prime; if $c\ge 2$, consider once again $A\colonequals\{1,3,5,\dots,2c-1\}$ with $|A|=c$, in which case
\[
I_A\ \subsetneq\ \big(x_1+x_2,\ x_3+x_4,\ \dots,\ x_{2c-1}+x_{2c}\big).
\]
Since $\height I_A=c$ by Theorem~\ref{theorem:n:large}\,\eqref{theorem:n:large:ci}, each ideal above has height $c$, so $I_A$ is not prime.

(3) Suppose $n=2c+2$, take $A\colonequals\{2,6,10,\dots,4c-2\}$. Then $|A|=c$ and $S/I_A$ is a normal domain of dimension $c+2$ by Theorem~\ref{theorem:n:large}\,\eqref{theorem:n:large:ufd}. It is however not a unique factorization domain: setting $i\colonequals\sqrt{-1}$ in $\CC$, the image of
\[
\big(x_1-ix_2,\ x_3-ix_4,\ \dots,\ x_{2c+1}-ix_{2c+2}\big)
\]
in $S/I_A$ is a height one prime ideal that is not principal.

(4) Quite generally, one has $\CC[e_1,\dots,e_n]=\CC[p_1,\dots,p_n]$ where $e_i$ is the $i$-th symmetric polynomial. Taking $n=2c-1$, it follows that
\[
p_{2c}\ \in\ \CC[p_1,\dots,p_{2c-1}]\equalscolon R.
\]
Degree considerations then imply that $p_{2c}=g_1p_1+\cdots+g_{c-1}p_{c-1}+g_cp_c^2$, where the $g_i$ are homogeneous elements of $R$. It follows that
\[
p_{2c}\ \in\ (p_1,\dots,p_{c-1},p_c^2)S
\]
where, recall, $S=\CC[x_1,\dots,x_n]$. Since $p_1,\dots,p_{c-1},p_{2c}$ is a regular sequence in the ring~$S$ by Theorem~\ref{theorem:n:large}\,\eqref{theorem:n:large:ci}, one has $p_{2c}\notin(p_1,\dots,p_{c-1})S$. Thus $g_c$, the coefficient of $p_c^2$ in the equation above, must be nonzero, hence a unit. It follows that
\[
p_c^2\ \in\ (p_1,\dots,p_{c-1},p_{2c})S.
\]
If $p_c\in(p_1,\dots,p_{c-1},p_{2c})S$, then degree considerations would force $p_c\in(p_1,\dots,p_{c-1})S$, which is not possible since $p_1,\dots,p_c$ is a regular sequence in $S$ by Theorem~\ref{theorem:n:large}\,\eqref{theorem:n:large:ci}. Hence,
taking $A\colonequals\{1,\dots,c-1,2c\}$ one has $p_c^2\in I_A$ and $p_c\notin I_A$, so the ideal $I_A$ is not radical.
\end{example}

\begin{proof}
The proofs of \eqref{theorem:n:large:ci} and \eqref{theorem:n:large:normal} are intertwined, using induction on $c$. Suppose $c=1$, then~\eqref{theorem:n:large:ci} is immediate, while \eqref{theorem:n:large:normal} follows using the Jacobian criterion for the hypersurface~$S/I_A$, bearing in mind that $n\ge 3$.

Next suppose $c>1$ and that $n\ge 2c-1$. By the inductive hypothesis, $S/(p_{a_1},\dots, p_{a_{c-1}})$ is a normal domain using \eqref{theorem:n:large:normal}, so \eqref{theorem:n:large:ci} follows. Suppose $n\ge 2c+1$ and that the elements of~$A$ are ordered as $a_1<\dots <a_c$. By induction we know that $p_A$ is a regular sequence; we determine the singular locus of $S/I_A$ using the Jacobian criterion:

Up to scalar multiples of the rows, the Jacobian matrix takes the form
\[
J\colonequals\begin{pmatrix}
x_1^{a_1-1} & x_2^{a_1-1} & \dots & x_n^{a_1-1}\\[0.4em]
x_1^{a_2-1} & x_2^{a_2-1} & \dots & x_n^{a_2-1}\\[0.1em]
\vdots & \vdots & & \vdots \\[0.1em]
x_1^{a_c-1} & x_2^{a_c-1} & \dots & x_n^{a_c-1}
\end{pmatrix}.
\]
Consider the size $c$ minors of the Jacobian matrix $J$ with respect to the lexicographic order induced by $x_n>x_{n-1}>\dots>x_1$, e.g., the minor determined by the first $c$ columns~is
\[
\det\begin{pmatrix}
x_1^{a_1-1} & x_2^{a_1-1} & \dots & x_c^{a_1-1}\\[0.4em]
x_1^{a_2-1} & x_2^{a_2-1} & \dots & x_c^{a_2-1}\\[0.1em]
\vdots & \vdots & & \vdots \\[0.1em]
x_1^{a_c-1} & x_2^{a_c-1} & \dots & x_c^{a_c-1}
\end{pmatrix}
\ =\ x_1^{a_1-1}x_2^{a_2-1}\cdots x_c^{a_c-1}+\text{ lower order terms}.
\]
Let $I_c(J)$ denote the ideal generated by the size $c$ minors of $J$, and let $H$ denote its initial ideal. Then $x_1^{a_1-1}x_2^{a_2-1}\cdots x_c^{a_c-1}\in H$, and similarly
\[
x_{i_1}^{a_1-1}x_{i_2}^{a_2-1}\cdots x_{i_c}^{a_c-1}\ \in\ H\qquad\text{ for all }1\le i_1<i_2<\dots<i_c\le n.
\]
Assume for the moment that $a_1\ge 2$, in which case each exponent $a_i-1$ above is positive. Then $\rad H$ contains each squarefree monomial of degree $c$ in the variables $x_1,\dots,x_n$, so $\height H\ge n-c+1$. On the other hand, if $a_1=1$, then $\rad H$ contains each squarefree monomial of degree $c-1$ in the $n-1$ variables $x_2,\dots,x_n$, so once again
\[
\height H\ \ge\ (n-1)-(c-1)+1\ =\ n-c+1.
\]
In either case the ideal $H$, and hence $I_c(J)$, has height at least $n-c+1$ in the polynomial ring $S$. It follows that in the ring $S/I_A$, the defining ideal of the singular locus has height at least $n-2c+1$. Under our assumption that $n\ge 2c+1$, the ring $S/I_A$ therefore satisfies the Serre condition $(R_v)$ with~$v=n-2c$, and is hence normal, completing the proof of \eqref{theorem:n:large:normal}.

In \eqref{theorem:n:large:ufd} one has $n\ge 2c+3$. If $c=0$ there is little to be said, so assume $c\ge 1$. Then~$S/I_A$ is a complete intersection ring of dimension at least $4$, satisfying the Serre condition~$(R_3)$ by the previous paragraph, and is hence a UFD by \cite[Corollaire~XI.3.14]{SGA2}.

For \eqref{theorem:n:large:radical}, note that $n\ge 2c$ implies that $S/I_A$ is a complete intersection, so our computation of the singular locus still applies, and shows that $S/I_A$ satisfies the Serre condition $(R_0)$.
\end{proof}

\begin{remark}
Suppose $n\ge 2c-1$, so that $I_A$ is a complete intersection of codimension $c$. Then, in the proof above, we saw that the ideal $I_c(J)$ has height at least $n-c+1$. As this is the upper bound for the height of the ideal of size $c$ minors of a $c\times n$ matrix, it follows that $\height I_c(J)=n-c+1$.

In a different direction, maximal minors of \emph{generalized Vandermonde matrices}
\[
\begin{pmatrix}
x_1^{b_1} & x_2^{b_1} & \dots & x_n^{b_1} \\[0.5em]
x_1^{b_2} & x_2^{b_2} & \dots & x_n^{b_2} \\[0.2em]
\vdots & \vdots & & \vdots \\[0.2em]
x_1^{b_c} & x_2^{b_c} & \dots & x_n^{b_c}
\end{pmatrix},
\]
where $c\ge n$, are studied in~\cite{FS}. Up to monomial and Vandermonde factors, these are the \emph{Schur polynomials}, see for example~\cite[page~76]{FS}.
\end{remark}

While Theorem~\ref{theorem:n:large} addresses the case of $c$ arbitrary power sums $p_{a_1},\dots, p_{a_c}$, we next record a result for consecutive power sums:

\begin{theorem}
\label{theorem:consecutive}
Set $S\colonequals\CC[x_1,\dots,x_n]$ be a polynomial ring, and let $a$ and $c$ be positive integers. Then the ring $S/(p_a,p_{a+1},\dots,p_{a+c-1})$ has an isolated singular point.
\end{theorem}

\begin{proof}
Set $R\colonequals S/(p_a,p_{a+1},\dots,p_{a+c-1})$. If $c\ge n$, then $R$ is an artinian local ring by \cite[Proposition~2.9]{CKW}, so the assertion is immediate. Assume $c<n$, in which case $R$ is a complete intersection ring by the same proposition; we examine the singular locus.

Up to scalar multiples of the rows, the Jacobian matrix takes the form
\[
J\colonequals\begin{pmatrix}
x_1^{a-1} & x_2^{a-1} & \dots & x_n^{a-1}\\[0.4em]
x_1^a & x_2^a & \dots & x_n^a\\[0.1em]
\vdots & \vdots & & \vdots \\[0.1em]
x_1^{a+c-2} & x_2^{a+c-2} & \dots & x_n^{a+c-2}
\end{pmatrix}.
\]
Using $I_c(J)$ for the ideal of minors as earlier, consider the ideal
\[
\fraka\colonequals I_c(J)+(p_a,p_{a+1},\dots,p_{a+c-1})S
\]
of $S$. It suffices to verify that the algebraic set $V(\fraka)$ contains no nonzero point of $\CC^n$. Suppose $\bsz\colonequals(z_1,\dots,z_n)\in V(\fraka)$. If $\bsz$ has at least $c$ distinct nonzero entries, without loss of generality $z_1,\dots,z_c$, evaluating the minor determined by the first $c$ columns of $J$ at $\bsz$ gives
\[
\det\begin{pmatrix}
z_1^{a-1} & z_2^{a-1} & \dots & z_c^{a-1}\\[0.4em]
z_1^a & z_2^a & \dots & z_c^a\\[0.1em]
\vdots & \vdots & & \vdots \\[0.1em]
z_1^{a+c-2} & z_2^{a+c-2} & \dots & z_c^{a+c-2}
\end{pmatrix}
\ =\ (z_1\cdots z_c)^{a-1}
\det\begin{pmatrix}
1 & 1 & \dots & 1\\[0.4em]
z_1 & z_2 & \dots & z_c\\[0.1em]
\vdots & \vdots & & \vdots \\[0.1em]
z_1^{c-1} & z_2^{c-1} & \dots & z_c^{c-1}
\end{pmatrix}
\]
which must be nonzero, a contradiction. It follows that the number $k$ of distinct entries of~$\bsz$ is at most $c$, allowing now for zero entries. Suppose $z_1,\dots,z_k$ are the distinct entries, and occur with multiplicity $m_1,\dots,m_k$ respectively in the $n$-tuple $\bsz$. The fact that the power sums $p_a,p_{a+1},\dots,p_{a+k-1}$ vanish at $\bsz$ gives us the matrix equation
\[
\begin{pmatrix}
1 & 1 & \dots & 1\\[0.4em]
z_1 & z_2 & \dots & z_k\\[0.1em]
\vdots & \vdots & & \vdots \\[0.1em]
z_1^{k-1} & z_2^{k-1} & \dots & z_k^{k-1}
\end{pmatrix}
\begin{pmatrix}
m_1z_1^a\\[0.4em]
m_2z_2^a\\
\vdots \\[0.1em]
m_kz_k^a
\end{pmatrix}
\ =\
\begin{pmatrix}
0\\[0.6em]
0\\
\vdots \\[0.1em]
0
\end{pmatrix}.
\]
This implies that the determinant of the Vandermonde matrix to the left must be zero, a contradiction. It follows that the only point in $V(\fraka)$ is $(0,\dots,0)$.
\end{proof}

\section{Power sums in four variables}

While each part of Theorem~\ref{theorem:n:large} is optimal in view of Example~\ref{example:n:large}, the boundary cases can be subtle and interesting; for example, when $n=4$ and $A=\{a,b\}$, the ideal $I_A$ is radical by Theorem~\ref{theorem:n:large}\,\eqref{theorem:n:large:radical}, but it appears difficult to determine when $I_A$ is prime, see Remark~\ref{remark:4:domain}. First, however, we record precisely when the ring $\CC[x_1,x_2,x_3,x_4]/(p_a, p_b)$ is a normal domain.

For $p$ a prime integer, let $\nu_p$ denote the $p$-adic valuation on $\ZZ\smallsetminus\{0\}$, i.e., $\nu_p(n)$ is the largest integer $e$ such that $p^e$ divides $n$.

\begin{theorem}
\label{theorem:4:normality}
Let $S\colonequals\CC[x_1,\dots,x_4]$. For positive integers $a<b$, set
\[
p_a\colonequals x_1^a+\dots+x_4^a
\qquad\text{and}\qquad
p_b\colonequals x_1^b+\dots+x_4^b.
\]
If $a=1$, then $S/(p_a,p_b)$ is a normal domain if and only if $b$ is even, whereas if $1<a<b$, then $S/(p_a,p_b)$ is a normal domain if and only if
\begin{enumerate}[\quad \rm(1)]
\item $\nu_2(a)\neq \nu_2(b)$, and
\item either $\nu_3(a)\neq \nu_3(b)$, or $\nu_3(a)=\nu_3(b)=\nu_3(a-b)$.
\end{enumerate}
\end{theorem}

\begin{proof}
Since $a$ and $b$ are distinct, $S/(p_a,p_b)$ is a complete intersection ring of dimension~$2$, and is normal precisely if the singular locus consists of a point. Set $\frakm$ to be the homogeneous maximal ideal of $S$.

Up to scalar multiples of the rows, the Jacobian matrix is
\[
\begin{pmatrix}
x_1^{a-1} & x_2^{a-1} & x_3^{a-1} & x_4^{a-1}\\[0.3em]
x_1^{b-1} & x_2^{b-1} & x_3^{b-1} & x_4^{b-1}
\end{pmatrix},
\]
with the ideal generated by its size two minors being
\[
\fraka\colonequals\big((x_ix_j)^{a-1}(x_j^{b-a}-x_i^{b-a})\ :\ 1\le i<j\le 4\big).
\]

Consider first the case where $a=1$. Then a minimal prime of $\fraka$ has the form
\[
\frakb\colonequals(x_1-\alpha x_4,\ x_2-\beta x_4,\ x_3-\gamma x_4),
\]
where $\alpha$, $\beta$, $\gamma$ are complex numbers with $\alpha^{b-1}=\beta^{b-1}=\gamma^{b-1}=1$. Since
\[
p_a\ \equiv\ (\alpha+\beta+\gamma+1)x_4\bmod\frakb,
\]
and
\[
p_b\ \equiv\ (\alpha^b+\beta^b+\gamma^b+1)x_4^b\ \equiv\ (\alpha+\beta+\gamma+1)x_4^b\bmod\frakb,
\]
it follows that $\frakm$ is the unique minimal prime of $\fraka+(p_a,p_b)$ unless there exist $\alpha$, $\beta$, $\gamma$ in~$\CC$ with $\alpha^{b-1}=\beta^{b-1}=\gamma^{b-1}=1$ and $\alpha+\beta+\gamma+1=0$. If $b$ is even, no such $(\alpha,\beta,\gamma)$ exists by Lemma~\ref{lemma:roots:of:unity}\,\eqref{lemma:roots:of:unity:3}, whereas if $b$ is odd, one may take $(\alpha,\beta,\gamma)$ to be $(-1,1,-1)$.

Next, suppose $a\ge 2$. Then, up to radical, the ideal $\fraka$ contains $x_ix_j(x_j^{b-a}-x_i^{b-a})$ for each~$1\le i<j\le 4$. It follows that, up to permuting indices, a minimal prime of $\fraka$ in $S$ has one of the following forms
\begin{enumerate}[\quad \rm(a)]
\item $(x_1,\ x_2,\ x_3)$,
\item $(x_1,\ x_2,\ x_3-\alpha x_4)$,
\item $(x_1,\ x_2-\alpha x_4,\ x_3-\beta x_4)$, or
\item $(x_1-\alpha x_4,\ x_2-\beta x_4,\ x_3-\gamma x_4)$,
\end{enumerate}
where $\alpha^{b-a}=\beta^{b-a}=\gamma^{b-a}=1$. We examine these in turn:

\medskip

Case (a). The only minimal prime of $(x_1,x_2,x_3)+(p_a,p_b)$ is $\frakm$.

\medskip

Case (b). The ideal $(x_1,x_2,x_3-\alpha x_4)+(p_a,p_b)$ has radical
\[
\big(x_1,\ x_2,\ x_3-\alpha x_4,\ (\alpha^a+1)x_4,\ (\alpha^b+1)x_4\big)\ =\
\big(x_1,\ x_2,\ x_3-\alpha x_4,\ (\alpha^a+1)x_4\big),
\]
where the equality above holds since $\alpha^{b-a}=1$. There exists such an ideal other than $\frakm$ precisely if $\nu_2(a)=\nu_2(b)$, see Lemma~\ref{lemma:roots:of:unity}\,\eqref{lemma:roots:of:unity:1}.

\medskip

Case (c). The ideal $(x_1,x_2-\alpha x_4,x_3-\beta x_4)+(p_a,p_b)$ has radical
\[
\big(x_1,\ x_2-\alpha x_4,\ x_3-\beta x_4,\ (\alpha^a+\beta^a+1)x_4\big).
\]
Use Lemma~\ref{lemma:roots:of:unity}\,\eqref{lemma:roots:of:unity:2}.

\medskip

Case (d). Lastly, the ideal $(x_1-\alpha x_4,x_2-\beta x_4,x_3-\gamma x_4)+(p_a,p_b)$ has radical
\[
\big(x_1-\alpha x_4,\ x_2-\beta x_4,\ x_3-\gamma x_4,\ (\alpha^a+\beta^a+\gamma^a+1)x_4\big),
\]
in which case we use Lemma~\ref{lemma:roots:of:unity}\,\eqref{lemma:roots:of:unity:3}.
\end{proof}

\begin{lemma}
\label{lemma:roots:of:unity}
Let $a$ and $b$ be distinct positive integers.
\begin{enumerate}[\quad \rm(1)]
\item\label{lemma:roots:of:unity:1} There exists $\alpha$ in $\CC$ with $\alpha^{b-a}=1$ and $\alpha^a+1=0$ if and only if $\nu_2(a)=\nu_2(b)$.

\item\label{lemma:roots:of:unity:2} There exists $\alpha$ and $\beta$ in $\CC$ with $\alpha^{b-a}=1=\beta^{b-a}$ and $\alpha^a+\beta^a+1=0$ if and only if~$\nu_3(a)=\nu_3(b)<\nu_3(b-a)$.

\item\label{lemma:roots:of:unity:3} There exists $\alpha$, $\beta$, and $\gamma$ in $\CC$ with $\alpha^{b-a}=\beta^{b-a}=\gamma^{b-a}=1$ and $\alpha^a+\beta^a+\gamma^a+1=0$ if and only if $\nu_2(a)=\nu_2(b)$.
\end{enumerate}
\end{lemma}

\begin{proof}
The conditions are symmetric with respect to $a$ and $b$, e.g., $\alpha^{b-a}=1$ gives $\alpha^b=\alpha^a$.

\eqref{lemma:roots:of:unity:1} If $e\colonequals\nu_2(a)=\nu_2(b)$, choose $\alpha$ with $\alpha^{2^e}=-1$, in which case $\alpha^a=-1=\alpha^b$. For the converse, let $a=2^ec$ and $b=2^fd$, where $c$ and $d$ are odd. If $\alpha^a=-1=\alpha^b$, then
\[
(\alpha^{cd})^{2^e}\ =\ -1\ =\ (\alpha^{cd})^{2^f},
\]
so $e=f$.

\eqref{lemma:roots:of:unity:2} Let $\omega$ be a primitive cube root of unity. If $e\colonequals\nu_3(a)=\nu_3(b)<\nu_3(b-a)$, choose~$\alpha$ with $\alpha^{3^e}=\omega$. Then $\alpha^{3^{e+1}}=1$, so $\alpha^{b-a}=1$. Setting $\beta\colonequals\alpha^2$, one has $\beta^{b-a}=1$ as well. Moreover, $\{\alpha^a, \beta^a\}=\{\omega, \omega^2\}$, so that
\[
\alpha^a+\beta^a+1\ =\ 0.
\]

For the converse, if $\alpha^a$ and $\beta^a$ are roots of unity with $\alpha^a+\beta^a+1=0$, then~$\alpha^a$ and~$\beta^a$ must be complex conjugates with real part $-1/2$. It follows that~$\{\alpha^a, \beta^a\}=\{\omega, \omega^2\}$. Assume, without loss of generality, that $\alpha^a=\omega$. Let $a=3^ec$ and $b=3^fd$, where $c$ and $d$ are relatively prime to $3$. Suppose now that $\alpha^{b-a}=1$. Then
\[
(\alpha^{cd})^{3^e}\ =\ \omega^d \quad\text{and}\quad (\alpha^{cd})^{3^f}\ =\ \omega^c
\]
are primitive cube root of unity, so $e=f$. Also $\alpha^{b-a}=1$ implies that $\alpha^{3^e(d-c)}=1$, so
\[
\omega^{d-c}\ =\ \alpha^{a(d-c)}\ =\ \alpha^{3^ec(d-c)}=1,
\]
implying that $3$ divides $d-c$.

\eqref{lemma:roots:of:unity:3} If $e\colonequals\nu_2(a)=\nu_2(b)$, choose~$\alpha$ with $\alpha^{2^e}=-1$. Then $\alpha^{2^{e+1}}=1$ so $\alpha^{b-a}=1$. Setting $\beta\colonequals\alpha^2$ and $\gamma\colonequals\alpha$, one has $\beta^{b-a}=\gamma^{b-a}=1$, and also
\[
\alpha^a+\beta^a+\gamma^a+1\ =\ (-1)+1+(-1)+1\ =\ 0.
\]

The converse: suppose $4$ distinct roots of unity sum to $0$, then the corresponding vectors in the complex plane have length $1$ and form a rhombus; pairing the parallel sides, each pair has sum~$0$. It follows that one of $\alpha^a$, $\beta^a$, or $\gamma^a$ equals $-1$. If the roots of unity are repeated, then $\{\alpha^a, \beta^a, \gamma^a, 1\}=\{\pm 1\}$. Assume, without loss of generality, that $\alpha^a=-1$. Then, if $\alpha^{b-a}=1$, part~\eqref{lemma:roots:of:unity:1} of the lemma implies that $\nu_2(a)=\nu_2(b)$.
\end{proof}

\begin{remark}
\label{remark:4:domain}
Set $S\colonequals\CC[x_1,x_2,x_3,x_4]$. It does not appear easy to determine precisely when the ring $S/(p_a, p_b)$ is a domain; we record some observations in this regard:

\begin{enumerate}[\quad \rm(1)]
\item If $a<b$ are odd integers, then $(p_a, p_b)$ is not prime since $(p_a, p_b)\subsetneq(x_1+x_2, x_3+x_4)$.

\item If $(p_a, p_b)$ is not prime, then neither is $(p_{ak}, p_{bk})$ for any positive integer $k$; one has an embedding of $\CC$-algebras $S/(p_a, p_b)\into S/(p_{ak}, p_{bk})$ induced by $x_i\mapsto x_i^k$.

\item If $b=4k+2$, then $S/(p_2, p_b)$ is not normal in view of Theorem~\ref{theorem:4:normality}. Moreover,
\[
(p_2, p_b)\ \subsetneq\ (x_1-ix_2,\ x_3-ix_4)
\]
shows that $(p_2, p_b)$ is not prime in this case.

\item When $a=2$, we conjecture that $S/(p_2, p_b)$ is a domain that is not normal precisely when $b=6k+5$ or $b=12k+8$, and $k$ is an integer with $k\ge1$. The case $k=0$ of these appears below:

\item The ideal $(p_2, p_5)$ is not prime: one has $p_5\in(p_1,p_2)$, see Remark~\ref{remark:codimension}, and it follows that $(p_2, p_5)\subsetneq(p_1,p_2)$.

\item The ideal $(p_2, p_8)$ is not prime: in the ring $S/(p_2, p_8)$ one has
\[
(x_2^2x_3^2 + x_2^2x_4^2 + x_3^2x_4^2 - x_1^4)^2 - 2(x_1x_2x_3x_4)^2\ =\ 0,
\]
so the image of $x_2^2x_3^2 + x_2^2x_4^2 + x_3^2x_4^2 - x_1^4 - \sqrt{2}\cdot x_1x_2x_3x_4$ in $S/(p_2, p_8)$ is a zerodivisor; one may verify readily that this image is nonzero.

In contrast, one may verify using \cite{Magma} or \cite{Macaulay2} that $\QQ[x_1,x_2,x_3,x_4]/(p_2, p_8)$ is an integral domain.

\item When $a=3$, we conjecture that $S/(p_3, p_b)$ is a domain that is not normal precisely when $b=18k+12$ and $k\ge 0$ is an integer.

\item We arrived at our conjectures in the cases $a=2$ and $a=3$ as follows: first one verifies using \cite{Magma} or \cite{Macaulay2} that when $\CC$ is replaced by $\QQ$, the corresponding ring
\[
R\colonequals\QQ[x_1,x_2,x_3,x_4]/(p_a, p_b)
\]
is an integral domain. Then we use the computational algebra programs to determine the integral closure $R'$ of $R$. Note that $R'\otimes_\QQ\CC$ is also normal, hence a product of normal domains. If ${[R']}_0=\QQ$, then $R'\otimes_\QQ\CC$ must be a normal domain, and it follows that its subring $R\otimes_\QQ\CC=S/(p_a, p_b)$ is a domain.
\end{enumerate}
\end{remark}

\section{Power sums in three variables: a special case of the conjecture}
\label{section:a=1}

We work over the complex numbers $\CC$ throughout this section. Given positive integers $a<b<c$ with~$\gcd(a,b,c)=1$, Conjecture~\ref{conjecture:6abc} as generalized in \cite[Conjecture~12]{MSW} may be rephrased as saying that the equations
\[
1+x^a+y^a\ =\ 1+x^b+y^b\ =\ 1+x^c+y^c\ =\ 0
\]
only have trivial solutions, i.e., with either $x$ and $y$ being cube roots of unity, or one of them being $0$ and the other being $-1$. We settle the conjecture when~$a=1$. In this case $y=-1-x$, so we are interested in solutions to the pair of polynomial equations
\begin{equation}
\label{equation:a=1}
1+x^b+(-1-x)^b \ =\ 0 \ =\ 1+x^c+(-1-x)^c.
\end{equation}
Indeed, we prove:

\begin{theorem}
\label{theorem:a=1}
For integers $b$ and $c$ with $1<b<c$, the only possible common zeros of the polynomials $1+x^b+(-1-x)^b$ and $1+x^c+(-1-x)^c$ are $0$, $-1$, $\omega$, $\omega^2$, where $\omega\colonequals e^{2\pi i/3}$. The common zeros at $0$, $-1$ occur when $2\nmid bc$, while the common zeros at $\omega$, $\omega^2$ occur when $3\nmid bc$. Consequently, when $6\mid bc$, there are no common zeros to the two polynomials.
\end{theorem}

Closely related problems were considered previously in \cite{Beukers,Nanninga}. In particular Beukers~\cite[Theorem 4.1]{Beukers} established the following result:

\begin{theorem}
If $\theta\in\CC$ differs from $0$, $-1$, $\omega$, $\omega^2$, where $\omega\colonequals e^{2\pi i/3}$, then there is at most one integer $n>1$ such that $1+\theta^n-(1+\theta)^n=0$.
\end{theorem}

If both $b$ and $c$ are odd, then Beukers's result shows that there are no solutions to~\eqref{equation:a=1} apart from $0$, $-1$, $\omega$, or $\omega^2$. We now treat the cases when at least one of $b$ or $c$ is even. Our proof has some points in common with Beukers's approach, but is also different in some details. When $b\le 5$ there are no roots of $1+x^b+(-1-x)^b$ apart from $0$, $-1$, $\omega$, $\omega^2$, and so we may assume in what follows that $b\ge 6$.

\begin{lemma}
\label{lemma:trivial:factors}
For integers $n\ge 2$, the polynomial $P_n(z)\colonequals 1+z^n+(-1-z)^n$ has degree $n$ if $n$ is even, and degree $n-1$ if $n$ is odd; it factors as $C_n(z)Q_n(z)$ where $C_n(z)$ equals
\begin{alignat*}2
1 &\quad\text{ for }n\equiv 0 \bmod 6;\\
z(z+1)(z^2+z+1)^2 &\quad\text{ for }n\equiv 1 \bmod 6;\\
(z^2+z+1) &\quad\text{ for }n\equiv 2 \bmod 6;\\
z(z+1) &\quad\text{ for }n\equiv 3 \bmod 6;\\
(z^2+z+1)^2&\quad\text{ for }n\equiv 4 \bmod 6;\\
z(z+1)(z^2+z+1)&\quad\text{ for }n\equiv 5 \bmod 6.
\end{alignat*}
In particular, the degree of $Q_n(z)$ is a multiple of six; the zeros of $Q_n(z)$ differ from $0$, $-1$, $\omega$, $\omega^2$ and occur in groups of six, with equal numbers of zeros on:
\begin{enumerate}[\quad \rm(1)]
\item the open line segments $\real(z) =-1/2$ going from $\omega$ to $-1/2+i\infty$, and its conjugate segment going from $\omega^2$ to $-1/2 - i\infty$;
\item the open arc of the unit circle going counterclockwise from $\omega$ to $\omega^2$;
\item the open arc of the circle $|z+1| = 1$ going counterclockwise from $\omega^2$ to $\omega$.
\end{enumerate}
Specifically, suppose $\alpha\colonequals -1/2 + it$ is a zero with $t>\sqrt{3}/2$. Then:
\begin{enumerate}[\quad \rm(i)]
\item $\alpha$ and $\bar{\alpha}=-1-\alpha$ are zeros on the conjugate line segments as above;
\item $\bar{\alpha}/\alpha=(-1-\alpha)/\alpha$ and $\alpha/\bar{\alpha}=-\alpha/(1+\alpha)$ are zeros lying on the arc of $|z|=1$;
\item $1/\alpha$ and $1/\bar{\alpha}$ are zeros lying on the arc of $|z+1|=1$.
\end{enumerate}
\end{lemma}

\begin{proof}
The first assertion on identifying the possible zeros at $0$, $-1$, $\omega$, $\omega^2$ is readily checked. We now produce the right number of zeros on the line segment $-1/2+it$ with $t>\sqrt{3}/2$ by counting sign changes; the remaining zeros will stem from these zeros $\alpha$ by taking $\bar{\alpha}$, $(-1-\alpha)/\alpha$, $-\alpha/(1+\alpha)$, $1/\alpha$ and $1/\bar{\alpha}$.

Write $z=-1/2+it$ as $z= -1/2 (1+i\tan\theta)= -e^{i\theta}/(2\cos\theta)$, where $\theta$ decreases from~$2\pi/3$ (when $z=-1/2+i\sqrt{3}/2$) to $\pi/2$ (when $z=-1/2+i\infty$). Note that $2\cos\theta$ goes from $-1$ to $0$ as $\theta$ decreases from $2\pi/3$ to $\pi/2$. Then
\[
P_n(z)\ =\ 1+2\cos(n\theta)/(-2\cos\theta)^n\ =\ \frac{2\cos(n\theta)+(2|\cos\theta|)^n}{(2|\cos\theta|)^n}.
\]
Clearly this is real valued, and has the same sign as the numerator, which is positive for values $\theta \in (\pi/2, 2\pi/3)$ with $n\theta \equiv 0 \bmod {2\pi}$, and negative for values $\theta \in (\pi/2, 2\pi/3)$ with~$n\theta \equiv \pi \bmod {2\pi}$. Upon splitting $n$ into progressions $\bmod \ 6$, and counting the sign changes produced in this way, we find that all the zeros of $P_n(z)$ are accounted for.
\end{proof}

Let $\calZ(b,c)$ denote the set of common zeros of the polynomials in~\eqref{equation:a=1}, excluding possible zeros at $0$, $-1$ or cube roots of unity. In other words, $\calZ(b,c)$ is the set of complex roots of $\gcd(Q_b(z),\ Q_c(z))$. We wish to show that this set is empty, and assume for the sake of contradiction that this is not the case. Naturally if $\alpha$ is a common zero, then so are all its Galois conjugates, as well as $1/\alpha$ (and its Galois conjugates), and~$(-1-\alpha)/\alpha$ together with its Galois conjugates. Let $\zeta$ denote an
element of $\calZ(b,c)$ of largest absolute value, and let $r$ denote this absolute value.

\begin{lemma}
\label{lemma:14:9}
Suppose that one of $b$ or $c$ is even. If $\calZ(b,c)$ is nonempty, then it contains an element with absolute value $r > 14/9$.
\end{lemma}

\begin{proof} Suppose to the contrary that $bc$ is even, and that all the elements in $\calZ(b,c)$ have absolute value bounded above by $14/9$. Consider the polynomial
\[
f(x)\colonequals\!\!\!\!\!\! \prod_{\alpha\in\calZ(b,c)} \!\!\!\!\!\! (x-\alpha).
\]
Note that $f(x)=\gcd(Q_b(x),\ Q_c(x))$ is a monic polynomial in $\QQ[x]$, and that it divides both $1+x^b +(-1-x)^b$ and $1+x^c+(-1-x)^c$. Since $b$ or $c$ is even, at least one of the polynomials $1+ x^b +(-1-x)^b$ or $1+x^c+(-1-x)^c$, that lie in $\ZZ[x]$, has leading coefficient~$2$. By unique factorization in $\ZZ[x]$, we conclude that $2f(x)$ must have integer coefficients.
Therefore $2f(\omega)$ is an element of~$\ZZ[\omega]$, and by the definition of $\calZ(b,c)$ we have $f(\omega) \neq 0$. It follows that
\[
2\!\!\!\!\!\! \prod_{\alpha \in \calZ(b,c)} \!\!\!\!\!\! |\omega -\alpha|\ =\ 2 |f(\omega)|\ \ge\ 1.
\]
Note that, as in Lemma~\ref{lemma:trivial:factors}, the zeros in $\calZ(b,c)$ occur in groups of $6$: if $\alpha =-\frac 12 +it$ lies in $\calZ(b,c)$, where $t > \sqrt{3}/2$, then so do $\bar{\alpha}$, $1/\alpha$, $1/\bar{\alpha}$, $-1-1/\alpha$, and $-1-1/\bar{\alpha}$. The contribution of such a group of $6$ to the product above is
\begin{multline*}
\Big|(\alpha-\omega)(\bar{\alpha}-\omega)(1/\alpha-\omega)(1/\bar{\alpha}-\omega)(\omega^2-1/\alpha)
(\omega^2-1/\bar{\alpha}) \Big| \\
\ =\ \frac{|\alpha^2+\alpha+1|^3}{|\alpha|^4}
\ =\ \frac{(t^2 -3/4)^3}{(1/4+t^2)^2}.
\end{multline*}
If $|\alpha| ={(1/4+t^2)}^{1/2} \le 14/9$, then the above is no greater than $0.4887\ldots < 1/2$, which gives a contradiction.
\end{proof}

Our next lemma treats the case when $c$ is small:

\begin{lemma}
\label{lemma:c:small}
Suppose that one of $b$ or $c$ is even and that $\calZ(b,c) \neq \emptyset$. Let $r$ be largest absolute value of an elements in $\calZ(b,c)$. Then $c$ must be larger than $\pi r^b/2$.
\end{lemma}

\begin{proof}
Let $\zeta\in\calZ(b,c)$ have maximal absolute value $r$. Since $1/\zeta$ must also be in $\calZ(b,c)$, we have
\[
\Big(-1-\frac{1}{\zeta^b}\Big)^c
\ =\ {\left[\Big(-1-\frac 1{\zeta}\Big)^b\right]}^c
\ =\ {\left[\Big(-1-\frac 1{\zeta}\Big)^c\right]}^b
\ =\ \Big(-1-\frac{1}{\zeta^c}\Big)^b.
\]
Taking logarithms, we see that
\begin{equation}
\label{equation:log}
\sum_{\ell=1}^{\infty} \frac{(-1)^{\ell-1}}{\ell} \Big(\frac{c}{\zeta^{b\ell}} - \frac{b}{\zeta^{c\ell}}\Big)\ \in\ \pi i \ZZ.
\end{equation}
However, by the triangle inequality, the quantity in~\eqref{equation:log} is bounded in absolute value by
\[
\sum_{\ell=1}^{\infty} \frac{1}{\ell} \Big(\frac{c}{r^{b\ell}} + \frac{b}{r^{c\ell}}\Big)\ \le\ \sum_{\ell=1}^{\infty} \frac{c + b/r}{r^{b\ell}}\ \le\ \frac{c(1+1/r)}{r^b-1}\ <\ \frac{2c}{r^b},
\]
since $r> 14/9$ by Lemma~\ref{lemma:14:9}. Thus, if $c \le \pi r^b/2$, then the quantity to the left in~\eqref{equation:log} is less than~$\pi$ in absolute value, so it must be zero.

But the triangle inequality also shows that the quantity in \eqref{equation:log} is bounded below in absolute value by
\begin{multline*}
\frac{c}{r^b} - \frac{b}{r^c} - \sum_{\ell=2}^{\infty}\frac{1}{\ell} \Big(\frac{c}{r^{b\ell}}+\frac{b}{r^{c\ell}}\Big)
\ >\
\frac{c}{r^b} - \frac{c}{r^c} - \sum_{\ell=2}^{\infty}\frac{c}{r^{b\ell}}\\
\ =\
\frac{c}{r^b} - \frac{c}{r^{c}} - \frac{c}{r^{b}(r^b-1)}
\ >\
\frac{c}{r^b}\Big( 1- \frac{1}{r} - \frac{1}{r^b-1}\Big).
\end{multline*}
Since $r >14/9$ and $b\ge 6$, the quantity above is strictly positive, and we have arrived at a contradiction. This proves the lemma.
\end{proof}

It remains to deal with the case when $c$ is large, specifically, $c > \pi r^b/2$. To handle this, we require a result on diophantine approximation due to Laurent, Mignotte, and Nesterenko \cite{LMN}; the formulation that we record below follows from \cite[Theorem~2.6]{Bugeaud} with a little cleaning up. By the \emph{primitive minimal polynomial} of an algebraic number $\alpha$ we mean the primitive polynomial $a_0 x^{d} + a_1 x^{d-1} + \dots + a_d \in \ZZ[x]$ of least degree with~$\alpha$ as a root, and $a_0$ a positive integer. In this case, the \emph{absolute height} of $\alpha$ is
\[
h(\alpha) \colonequals \frac 1d \Big(\log a_0 + \sum_\sigma\log\max\left\{1,\ |\sigma(\alpha)|\right\}\Big),
\]
where the elements $\sigma(\alpha)$ are the Galois conjugates of $\alpha$.

\begin{lemma}
\label{lemma:LMN}
Let $\alpha$ be an algebraic number of absolute value $1$ that is not a root of unity, and let $d$ be its degree. Let $h(\alpha)$ denote the absolute height of $\alpha$ as above. Then, for any positive integer $k$, we have
\[
|\alpha^k -1|\ \ge\ \exp\Big(-\frac{9}{8} \big(22\pi + dh(\alpha)\big) \big(\max\left\{34,\ d\log (k/2) + 10\right\}\big)^2 \Big).
\]
\end{lemma}

\begin{proof}[Proof of Theorem~\ref{theorem:a=1}]
Let $\zeta$ be an element of the set $\calZ(b,c)$ with maximal absolute value $r\colonequals|\zeta|$, and take $\alpha = -1 -1/\zeta$, so that $\alpha$ is an element of $\calZ(b,c)$ with $|\alpha|=1$. Note that $\alpha$ cannot be a root of unity, else some conjugate of $\alpha$ will not lie on the arc from $\omega$ to~$\omega^2$. Since $\alpha$ is a root of $1+x^b+(-1-x)^b$, the degree $d$ of $\alpha$ is at most $b$. Since one of $b$ or $c$ is even,~$\alpha$ satisfies a polynomial in $\ZZ[x]$ with leading coefficient $2$, so that the primitive minimal polynomial of $\alpha$ in $\ZZ[x]$ has leading coefficient $1$ or $2$. Since only one third of the elements of $\calZ(b,c)$ have absolute value exceeding $1$, and these absolute values are bounded above by $r$, we conclude that
\[
dh(\alpha)\ \le\ \log 2 + \frac{b}{3} \log r.
\]
Appealing to Lemma~\ref{lemma:LMN}, we conclude that for any positive integer $k$ one has
\begin{equation}
\label{equation:LMN:1}
|\alpha^k -1|\ \ge\ \exp\Big(-\frac{9}{8} \big(70+ \frac b3 \log r\big) \big(\max\left\{34,\ b\log (k/2) + 10\right\}\big)^2 \Big).
\end{equation}
Since $\alpha$ is a root of $1+x^c +(-1-x)^c$, and $|-1-\alpha| = 1/r$, we have
$|1 + \alpha^c |\le 1/r^c$ so
\begin{equation}
\label{equation:LMN:2}
|\alpha^{2c}-1|\ \le\ \frac{2}{r^c}.
\end{equation}

On the other hand, assuming that $c\ge e^5$ and using that $b\ge 6$, we may simplify the bound in~\eqref{equation:LMN:1} to yield
\begin{multline*}
|\alpha^{2c}-1|
\ \ge\
\exp\Big(-\frac{9}{8}\big(70+\frac{b}{3}\log r\big) (b\log c +10)^2 \Big)\\
\ \ge\
\exp\Big( - 2 b^2 (\log c)^2 \big(70+\frac b3\log r \big)\Big).
\end{multline*}
Comparing this with~\eqref{equation:LMN:2}, we obtain a contradiction unless
\[
c\log r\ \le\ \log 2 + 2b^2 (\log c)^2 \big( 70 + \frac{b}{3} \log r\big).
\]
Since $r>14/9$ by Lemma~\ref{lemma:14:9}, the above bound, under the assumption $c\ge e^5$, implies that
\begin{equation}
\label{equation:LMN:3}
\frac{c}{(\log c)^2}
\ \le\
\frac{\log 2}{\log (14/9) (\log c)^2} + 2b^2 \Big(\frac{70}{\log (14/9)} + \frac b3\Big)
\ \le\
320 b^2 + 2b^3/3.
\end{equation}

If $b\ge 43$, then by Lemmas~\ref{lemma:14:9} and~\ref{lemma:c:small} we see that $\calZ(b,c)=\emptyset$ unless $c \ge (\pi/2)(14/9)^b$. But a small calculation shows that this lower bound for $c$, which is much bigger than~$e^5$, contradicts the upper bound imposed in \eqref{equation:LMN:3}. Thus we conclude that $\calZ(b,c) = \emptyset$ whenever $c > b \ge 43$.

For $6\le b\le 42$ it is easy to check that after accounting for the zeros at $0$, $-1$, $\omega$, $\omega^2$, the remaining part of the polynomial $1+x^b +(-1-x)^b$, denoted earlier by $Q_b(x)$, is irreducible. This allows us to obtain improved estimates for the size of $r$ in Lemma~\ref{lemma:14:9}, thereby obtaining a larger lower bound for $c$ in Lemma~\ref{lemma:c:small}. For all $17\le b\le 42$, the polynomial $1+x^b + (-1-x)^b$ has a root of size at least $2.72$, so that in these cases we may use $r\ge 2.72$, and $c\ge (\pi/2) (2.72)^b$; this bound can be checked to contradict \eqref{equation:LMN:3}. Thus $\calZ(b,c) = \emptyset$ for $c> b\ge 17$. When~$b$ equals~$12$, $14$, or $16$, there is a root of size $r \ge 3.83$, and our argument applies in these cases as well.

The case $b=6$ is covered by \cite[Theorem~2.11]{CKW}, while the case $b=7$ does not arise, since $1+x^7+(-1-x)^7$ only has roots at $0$, $-1$, $\omega$, $\omega^2$. When $b=9$, the nontrivial factor of~$1+x^9+(-1-x)^9$ is a primitive irreducible polynomial of degree $6$, with leading coefficient $3$, and therefore cannot divide $1+x^c+(-1-x)^c$ for $c$ even, since this polynomial has leading coefficient $2$. Similarly, when $b=15$, the nontrivial factor of $1+x^{15}+(-1-x)^{15}$ is a primitive irreducible polynomial of degree $12$, with leading coefficient $15$, and once again this cannot divide $1+x^c+(-1-x)^c$ for $c$ even.

We are left with four remaining cases, $b=8$, $10$, $11$, and $13$, where an additional small computation is needed to check the theorem. We illustrate this calculation in the case $b=8$, the other cases being similar. The nontrivial factor of $1+x^8+(-1-x)^8$ has degree $6$, with a root of largest absolute value at
\[
\zeta \approx -\frac 12 + 2.513228157188 i.
\]
It follows from Lemma~\ref{lemma:c:small} that $\calZ(8,c)=\emptyset$ for $8 < c \le 2500$, while from \eqref{equation:LMN:3} it follows that $\calZ(8,c)=\emptyset$
for $c>5 \times 10^6$. To handle the remaining range for $c$, write $(1+1/\zeta^8)$ as~$e^{i\theta}$ with $\theta = -0.0005379141\ldots$, so that by \eqref{equation:log} we have, for some integer $m$,
\[
|c\theta + m\pi|\ \le\ 8\sum_{\ell=1}^{\infty}\frac{1}{\ell |\zeta|^{c\ell}}\ \le\ 9\times(2.5)^{-c}\ <\ (2.5)^{-2400}.
\]
Thus $m\pi/|\theta|$ must be extremely close to the integer $c$. Now
\[
\pi/|\theta| = 5840.32375784959\ldots,
\]
and since $2500 < c \le 5\times 10^6$, we may restrict attention to integers $m$ that lie in the range $1\le m\le 1000$. A rapid calculation (for instance by examining the continued fraction expansion of $\pi/|\theta|$) shows that there are no $m$ in this range with $m\pi/|\theta|$ being extremely close to an integer, which completes our treatment of the case $b=8$.
\end{proof}

\section{Power sums in three variables: the general case}
\label{section:abc:general}

Adapting the argument from the previous section, we establish the more general result:

\begin{theorem}
\label{theorem:abc:almost:all}
Let $2 \le a < b <c$ be integers with $2\mid abc$, and $\gcd(a,b,c)=1$. Suppose that the system of equations 
\[
1+x^a+y^a\ =\ 1+x^b+y^b\ =\ 1+x^c+y^c\ =\ 0
\]
has a solution where $x$ and $y$ are not cube roots of unity. Then:
\begin{enumerate}[\quad \rm(1)]
\item\label{theorem:abc:almost:all:1} We have $b < 600 a^2 2^a$. 

\item\label{theorem:abc:almost:all:2} If exactly one of $a$, $b$, $c$ is even, then $b < 600 a^2$. 

\item\label{theorem:abc:almost:all:3} For each $b$ in the range $a< b < 600 a^2 2^a$, there are at most finitely many possible choices for $c$. 
\end{enumerate}
\end{theorem}

Let $\calZ(a,b,c)$ denote the set of all $\alpha\in\CC$, excluding cube roots of unity, for which there exists some $\beta\in\CC$ with
\[
1+\alpha^a+\beta^a\ =\ 1+\alpha^b+\beta^b\ =\ 1+\alpha^c+\beta^c\ =\ 0.
\]

\begin{lemma}
\label{lemma:leading:coefficient}
Suppose that $\gcd(a,b,c)=1$ and that at least one of $a$, $b$, or $c$ is even. If $\alpha \in \calZ(a,b,c)$, then the primitive minimal polynomial of $\alpha$ in $\ZZ[x]$ has degree at most $ab$, and leading coefficient $1$ or $2$. If exactly one of $a$, $b$, or $c$ is even, then the leading coefficient must be $1$, i.e., $\alpha$ is an algebraic integer.
\end{lemma}

\begin{proof}
Note that
\[
(1+\alpha^a)^b\ =\ (-\beta^a)^b\ =\ (-1)^b (\beta^b)^a\ =\ (-1)^{b+a} (1+\alpha^b)^a,
\]
and similarly $(1+\alpha^a)^c = (-1)^{a+c} (1+\alpha^c)^a$, and $(1+\alpha^b)^c = (-1)^{b+c} (1+\alpha^c)^b$. Thus, $\alpha$ is a root of the three polynomials
\begin{multline}
\label{eqn:lemma:leading:coefficient}
(1+x^a)^b - (-1)^{a+b}(1+x^b)^a,\quad (1+x^a)^c - (-1)^{a+c} (1+x^c)^a,\\
\text{and} \quad (1+x^b)^c - (-1)^{b+c}(1+x^c)^b.
\end{multline}
It follows that $\alpha$ is an algebraic number of degree at most $ab$. Furthermore, since two of the integers $a$, $b$, $c$ must have opposite parity, one of the displayed polynomials must have leading coefficient~$2$, so the primitive minimal polynomial for $\alpha$ must have leading coefficient $1$ or $2$. Finally, if exactly one of $a$, $b$, $c$ is even, then two of the three polynomials have leading coefficient $2$, and the third has an odd leading coefficient. Therefore, in this case, the primitive minimal polynomial of $\alpha$, which divides all three of the polynomials~\eqref{eqn:lemma:leading:coefficient}, has leading coefficient $1$.
\end{proof}

\begin{lemma}
\label{lemma:10:delta}
Suppose $w$ is a complex number with $e^{-\delta} \le |w| \le e^{\delta}$ and $e^{-\delta} \le |1+w| \le e^{\delta}$, where $0 \le \delta \le 1/10$. Then
\[
|w^2+w+1|\ \le\ 10\delta.
\]
\end{lemma}

\begin{proof}
By assumption,
\[
|1+w|^2\ =\ 1+w+\bar{w}+|w|^2
\]
lies in the interval $[e^{-2\delta},\ e^{2\delta}]$, so that
\[
|1+w+\bar{w}|\ \le\ \max\{e^{2\delta}-|w|^2,\ |w|^2-e^{-2\delta}\}\ \le\ e^{2\delta}-e^{-2\delta}.
\]
Therefore
\begin{multline*}
|w^2+w+1|\ =\ |w|\Big|w+\frac{1}{w}+1\Big|\ \le\ |w|\Big(\Big|w+\bar{w}+1\Big|+\Big|\frac{1}{w}-\bar{w}\Big|\Big)
\\
\le\ |w|(e^{2\delta}-e^{-2\delta})+\Big|1-|w|^2\Big|\ \le\ e^{\delta}(e^{2\delta}-e^{-2\delta}) + (e^{2\delta}-1),
\end{multline*}
and the lemma follows.
\end{proof}

\begin{lemma}
\label{lemma:largest:absolute}
Suppose $\gcd(a,b,c)=1$ and $2\mid abc$. Suppose $\calZ(a,b,c) \neq \emptyset$, let $r$ denote the largest absolute value of an element of $\calZ(a,b,c)$. Then
\[
r\ \ge\ \exp\Big(\frac{1}{10a2^a}\Big).
\]
If exactly one of $a$, $b$, $c$ is even, then this may be improved to
\[
r\ \ge\ \exp\Big(\frac{1}{10a}\Big).
\]
\end{lemma}

\begin{proof}
Note that if $\alpha$ belongs to $\calZ(a,b,c)$, then so does $1/\alpha$. Thus all elements of $\calZ(a,b,c)$ have absolute value
between $1/r$ and $r$.

For $\alpha \in \calZ(a,b,c)$, let $\beta$ be such that $1+\alpha^a+\beta^a =1+\alpha^b+\beta^b= 1+\alpha^c+\beta^c=0$. We know that $\alpha^a$ and $\beta^a$ both have absolute value in the interval $[r^{-a},\ r^a]$. But $\beta^a = -(1+\alpha^a)$, so by Lemma~\ref{lemma:10:delta} we conclude that
\begin{equation}
\label{equation:largest:absolute}
|\alpha^{2a}+\alpha^a+1|\ \le\ 10\log(r^a).
\end{equation}

Next, we claim that $\alpha^{2a}+\alpha^a+1$ cannot equal zero. If it did, then $\alpha^a$ would be a primitive cube root of unity, i.e., $\omega$ or $\omega^2$, and therefore so would $\beta^a$. Now $\alpha^b$ and $\beta^b= -(1+\alpha^b)$ both have absolute value $1$, so that by Lemma~\ref{lemma:10:delta} $\alpha^b$ must be $\omega$ or $\omega^2$. The same conclusion holds for $\alpha^c$. But since $\gcd(a,b,c)=1$, we conclude that $\alpha$ itself must be a cube root of unity, which is not permitted given the definition of $\calZ(a,b,c)$.

Summarizing the argument thus far, if $\alpha\in\calZ(a,b,c)$ then $\alpha$ and all its Galois conjugates satisfy the bound~\eqref{equation:largest:absolute}, and furthermore $\alpha^{2a}+\alpha^a+1 \neq 0$. Let $f(x)$ denote the primitive minimal polynomial for $\alpha$ in $\ZZ[x]$, and set $g(x)\colonequals x^{2a}+x^a+1$. By Lemma~\ref{lemma:leading:coefficient}, the degree $d$ of $f(x)$ is at most $ab$, and its leading coefficient is $1$ or $2$. The resultant of~$f(x)$ and $g(x)$ is a nonzero integer, and therefore
\[
1\ \le\ |\Res(f,g)|\ \le\ 2^{2a}\prod_{\sigma}\Big|\sigma(\alpha)^{2a}+\sigma(\alpha)^{a}+1\Big|\ \le\ 2^{2a}(10a\log r)^{d},
\]
where $\sigma(\alpha)$ are the Galois conjugates of $\alpha$, and we have used~\eqref{equation:largest:absolute} for the upper bound. Since $d$ must be at least $2$, the first bound of the lemma follows. If exactly one of $a$, $b$, $c$ is even, then $f(x)$ is monic, and the improved bound holds.
\end{proof}

\begin{lemma}
\label{lemma:lower:bound:c}
Suppose $\gcd(a,b,c)=1$ and $2\mid abc$. Suppose $\calZ(a,b,c) \neq \emptyset$, let $r$ be the largest absolute value of an element of $\calZ(a,b,c)$. Then $c$ must be larger than $\pi r^b/2$.
\end{lemma}

\begin{proof}
The argument is identical to the proof of Lemma~\ref{lemma:c:small}.
\end{proof}

\begin{lemma}
\label{lemma:smallest:absolute} Suppose $\gcd (a,b,c)=1$ and $2\mid abc$. 
Let $\alpha$ denote an element of $\calZ(a,b,c)$ with smallest absolute value, which is~$1/r$. Let $\beta$ be such that
\[
1+\alpha^a+\beta^a\ =\ 1+\alpha^b+\beta^b\ =\ 1+\alpha^c+\beta^c\ =\ 0.
\]
Then $\zeta\colonequals\beta/\bar{\beta}$ is an algebraic number of degree at most $(ab)^2$, with absolute height
\[
h(\zeta)\ \le\ 2 \log(2r).
\]
If $2b^8 \le r^b$, then $\zeta$ is not a root of unity. If $\zeta$ is a root of unity, then either $r^c < 2b^8$, or $\alpha^c$ and $\beta^c$ are both real numbers.
\end{lemma}

\begin{proof}
Since $\beta$ is an algebraic number with degree at most $ab$ by Lemma~\ref{lemma:leading:coefficient}, it follows that $\zeta=\beta/\bar{\beta}$ has degree at most $(ab)^2$. As $\beta$ has a primitive minimal polynomial with leading coefficient at most $2$, and since all its Galois conjugates have absolute value at most $r$, we see that $h(\beta) \le \log (2r)$. Now
\[
h(\zeta)\ =\ h(\beta/\bar{\beta})\ \le\ h(\beta) + h(\bar{\beta})\ \le\ 2\log (2r).
\]

It remains to justify the assertions about when $\zeta$ can be a root of unity. Suppose that it is, write $\beta=|\beta|e^{\pi i\ell/k}$ where $\ell/k$ is a reduced fraction. Then $\zeta=e^{2\pi i\ell/k}$ is a primitive $k$-th root of unity.

Suppose that $b$ is not a multiple of $k$. Then
\begin{multline*}
r^{-2b}\ =\ |1+\beta^b|^2\ =\ 1+|\beta|^{2b}+2|\beta|^b\cos(\pi\ell b/k)\ \ge\ (1+|\beta|^{2b})(1-|\cos(\pi\ell b/k)|)\\
\ge\ (1-\cos(\pi/k))\ >\ k^{-2},
\end{multline*}
so that $k > r^{b}$. However the degree of $\zeta$ is $\phi(k)$, which is at most $(ab)^2$. Now $\phi(k) \ge \sqrt{k/2}$ for all integers $k$, so
\[
r^b\ <\ k\ \le\ 2\phi(k)^2\ \le\ 2(ab)^4\ <\ 2b^8. 
\]
In other words, if $r^b \ge 2b^8$ then $b$ must be a multiple of $k$. The same argument shows that if $r^c \ge 2b^8$ then $c$ is a multiple of $k$.

If $b$ is a multiple of $k$, then $\beta^b$ is real, which forces $\alpha^b$ to also be real. Similarly, if $c$ is a multiple of $k$, then $\beta^c$ and $\alpha^c$ are once again real numbers. The last assertion of the lemma is immediate. 

Finally if $r^b \ge 2b^8$, then our argument so far shows that $b$ and $c$ are multiples of $k$. Now we must have $|\beta|^b = 1+\alpha^b$, and $|\beta|^c = 1+\alpha^c$, so that $\alpha^b$ and $\alpha^c$ must be real numbers (of absolute value $r^{-b}$ and $r^{-c}$ respectively). If $|\beta| \ge 1$, then $\alpha^b = r^{-b}$ and $\alpha^c = r^{-c}$. However,
\[
|\beta|^c\ \ge\ |\beta|^b\ =\ 1+r^{-b}\ >\ 1+r^{-c}\ =\ |\beta|^c
\]
yields a contradiction. Similarly, if $|\beta|<1$, then $\alpha^b = -r^{-b}$ and $\alpha^c = -r^{-c}$, and
\[
|\beta|^b\ >\ |\beta|^c\ =\ 1-r^{-c}\ >\ 1-r^{-b}\ =\ |\beta|^b
\]
gives a contradiction. Thus, in this situation $\zeta$ cannot be a root of unity, and this completes the proof of the lemma.
\end{proof}

\begin{proof}[Proof of Theorem~\ref{theorem:abc:almost:all}]
We begin by proving the first two parts of the theorem. We assume that $\calZ(a,b,c) \neq \emptyset$, and note that Lemma~\ref{lemma:largest:absolute} gives a lower bound for the largest absolute value $r$ of an element of $\calZ(a,b,c)$. We assume that $b$ is at least $600 a^2 2^a$ or $600a^2$, depending on whether we seek to establish~\eqref{theorem:abc:almost:all:1} or ~\eqref{theorem:abc:almost:all:2}, and work towards a contradiction. Using the lower bounds for $r$ from Lemma~\ref{lemma:largest:absolute} in the respective cases, we see that $r^b \ge 2 b^8$. Hence, taking $\alpha$, $\beta$, $\zeta$ as in Lemma~\ref{lemma:smallest:absolute}, we see that $\zeta$ is not a root of unity. Since $\beta^c = -(1+\alpha^{c})$, we have
\[
\zeta^c\ =\ \frac{\beta^c}{\bar{\beta}^c}\ =\ \frac{1+\alpha^c}{1+\bar{\alpha}^c},
\]
and so
\begin{equation}
\label{equation:final:proof:1}
|\zeta^c-1|\ \le\ \frac{2r^{-c}}{1-r^{-c}}\ \le\ 3 r^{-c}
\end{equation}
since $r^c>r^b>3$. On the other hand, from Lemma~\ref{lemma:LMN} and Lemma~\ref{lemma:smallest:absolute} we know that
\[
|\zeta^c-1|\ \ge\ \exp\Big(-\frac{9}{8}\big(70+(ab)^2 2\log(2r)\big)(ab)^4(\log c)^2\Big).
\]
Since $ab\ge 100$, we may simplify the above to
\[
|\zeta^{c}-1|\ \ge\ \exp\Big(-(ab)^6 (2+3\log r) (\log c)^2\Big).
\]
Combining this with~\eqref{equation:final:proof:1}, we conclude that
\begin{equation}
\label{equation:final:proof:2}
\frac{c}{(\log c)^2}\ \le\ 3(ab)^6 \Big(1+\frac{1}{\log r}\Big).
\end{equation}

On the other hand, $c\ge \pi r^b/2$ by Lemma~\ref{lemma:lower:bound:c}. Since $r^b \ge 10$, we have $c/(\log c)^2 \ge r^b/(b\log r)^2$, which along with \eqref{equation:final:proof:2} gives
\[
r^b\ \le\ 3 a^6 b^8 \log r (1 +\log r).
\]
Since $b\ge 600 a^2$, we find
\begin{equation*} 
r^{b/2}\ \ge\ \frac{(b\log r)^{11}}{2^{11} \cdot 11!}\ \ge\ \frac{b^8 (\log r)^{11}}{2^{11} \cdot 11!} (600a^2)^3
\ >\ a^6 b^8 \frac{(\log r)^{11}}{380},
\end{equation*} 
and combining this with our upper bound on $r^b$, we conclude that
\[
r^{b/2}\ <\ 1140(\log r)^{-10} (1 + \log r).
\]
In other words,
\[
b\ <\ \frac{2}{\log r}\log\big(1140 (\log r)^{-10} (1 + \log r)\big). 
\]
Inserting here the bounds from Lemma~\ref{lemma:largest:absolute} which give $\log r \ge (10 a 2^{a})^{-1}$ in case~\eqref{theorem:abc:almost:all:1} and $\log r \ge (10a)^{-1}$ in case~\eqref{theorem:abc:almost:all:2}, we obtain the desired contradiction. 

It remains lastly to establish~\eqref{theorem:abc:almost:all:3}. Fix $a$ and $b$ with $2\le a < b < 600 a^2 2^a$. We wish to show that if $c$ is sufficiently large, with $2\mid abc$ and $\gcd(a,b,c)=1$, then $\calZ(a,b,c) = \emptyset$. First note that any $\alpha\in\calZ(a,b,c)$ is a root of the polynomial
\[
(1+x^a)^b -(-1)^{a+b} (1+x^b)^a
\]
by~\eqref{eqn:lemma:leading:coefficient}, and thus lies in a set of size at most $ab$. Let $\alpha$, $\beta$, $\zeta$, and $r$ be as in Lemma~\ref{lemma:smallest:absolute}, and assume that $c \ge 600 a^2 2^a$ so that $r^c \ge 2 c^8 \ge 2b^8$. If $\zeta$ is not a root of unity, then our earlier argument invoking Lemma~\ref{lemma:LMN} applies, and yields the upper bound~\eqref{equation:final:proof:2}, which shows that there are at most finitely many possibilities for $c$. 
Finally, if $\zeta$ is a root of unity, then the last assertion of Lemma~\ref{lemma:smallest:absolute} yields that $\alpha^c$ and $\beta^c$ are real with $1+\alpha^c+\beta^c=0$. Since $|\alpha|= r^{-1} <1$, this equation may be written as $|\beta|^c = 1 + r^{-c}$ if $|\beta| >1$, and as $|\beta|^c = 1 - r^{-c}$ if $|\beta| <1$. Given $\alpha$ and $\beta$, there can be at most one solution $c$ to these equations. Finally, since $\alpha$ and $\beta$ are elements of the finite set of roots of the polynomial $(1+x^a)^b -(-1)^{a+b} (1+x^b)^a$, there are only finitely many possibilities for $c$. 
\end{proof}

\section*{Acknowledgments}

The use of the computer algebra systems \texttt{Macaulay2}~\cite{Macaulay2} and \texttt{Magma}~\cite{Magma} is gratefully acknowledged.


\end{document}